\documentclass[11pt]{amsart}

\usepackage{etex}
\usepackage{amsmath, amssymb}
\usepackage{array}
\usepackage[frame,cmtip,arrow,matrix,line,graph,curve]{xy}
\usepackage{graphpap, color, paralist, pstricks}
\usepackage[mathscr]{eucal}
\usepackage[pdftex]{graphicx}
\usepackage[pdftex,colorlinks,backref=page,citecolor=blue]{hyperref}
\usepackage{tikz}
\usepackage{kotex}
\usepackage{mathtools}

\newtheorem{theorem}{Theorem}[section]
\newtheorem{proposition}[theorem]{Proposition}
\newtheorem{corollary}[theorem]{Corollary}
\newtheorem{lemma}[theorem]{Lemma}

\theoremstyle{definition}
\newtheorem{definition}[theorem]{Definition}

\newtheorem{question}[theorem]{Question}
\newtheorem{remark}[theorem]{Remark}

\newtheorem{notation}[theorem]{Notation}

\numberwithin{equation}{section}

\newcommand{\PP}{\mathbb{P}}

\newcommand{\CC}{\mathbb{C}}

\newcommand{\cO}{\mathcal{O} }

\newcommand{\cC}{\mathcal{C} }

\newcommand{\cH}{\mathcal{H} }

\newcommand{\cM}{\mathcal{M} }

\newcommand{\cS}{\mathcal{S} }

\newcommand{\cV}{\mathcal{V} }
\newcommand{\cU}{\mathcal{U} }

\newcommand{\cZ}{\mathcal{Z} }

\newcommand{\rE}{\mathrm{E} }
\newcommand{\rIE}{\mathrm{IE} }
\newcommand{\rH}{\mathrm{H} }

\newcommand{\rP}{\mathrm{P} }
\newcommand{\rIP}{\mathrm{IP} }

\newcommand{\bH}{\mathbf{H} }

\newcommand{\bM}{\mathbf{M} }

\newcommand{\bS}{\mathbf{S} }

\def\Hom{\mathrm{Hom} }
\def\Ext{\mathrm{Ext} }

\def\Gr{\mathrm{Gr} }

\def\git{/\!/ }

\def\lr{\rightarrow}

\newcommand{\ses}[3]{0\rightarrow{#1}\rightarrow{#2}\rightarrow{#3}\rightarrow0}

\newcommand{\tcite}[2]{\protect{\cite[{#1}]{#2}}}

\def\mapright#1{\,\smash{\mathop{\longrightarrow}\limits^{#1}}\,}

\makeatletter
\providecommand{\leftsquigarrow}{%
  \mathrel{\mathpalette\reflect@squig\relax}%
}
\newcommand{\reflect@squig}[2]{%
  \reflectbox{$\m@th#1\rightsquigarrow$}%
}
\makeatother

\begin{document}

\title[Desingularization of twisted cubics on $V_5$]{Desingularization of Kontsevich's compactification of twisted cubics in $V_5$}
\author{Kiryong Chung}
\address{Department of Mathematics Education, Kyungpook National University, 80 Daehakro, Bukgu, Daegu 41566, Korea}
\email{krchung@knu.ac.kr}

\date{}
\begin{abstract}
By definition, the del Pezzo $3$-fold $V_5$ is the intersection of $\mathrm{Gr}(2,5)$ with three hyperplanes in $\mathbb{P}^9$ under Pl\"ucker embedding, and rational curves in $V_5$ have been examined in various studies on Fano geometry. In this paper, we propose an explicit birational relation for the Kontsevich and Simpson compactifications of twisted cubic curves in $V_5$. As a direct corollary, we obtain a desingularized model of Kontsevich compactification that induces the intersection cohomology group of Kontsevich's space.
\end{abstract}
\maketitle
\section{Introduction}\label{sec:intro}
\subsection{Compactified moduli spaces of rational curves} Let $X\subset \PP^r$ be a smooth projective variety with a fixed polarization $\cO_X(1)$. The compactified moduli spaces of interest to us are the following:
\begin{itemize}
\item \emph{Kontsevich space}:
Let $C$ be a projective connected reduced curve. A map $f: C \to X$ is considered \emph{stable} if $C$ has at worst nodal singularities and $|\mathrm{Aut}(f)|<\infty$.
Let $\cM(X,d)$ be the moduli space of isomorphism classes of stable
maps $f:C\to X$ with genus $g(C)=0$ and $\mathrm{deg} f^*\cO_X(1)=d$.
\item \emph{Simpson space}:
For a coherent sheaf $F\in \mathrm{Coh}(X)$, the Hilbert polynomial $P(F)(m)$ is defined by $\chi(F\otimes \cO_X(m))$. If the support of $F$ has dimension $n$, $P(F)(m)$ has degree $n$ and it can be written as
\[
	P(F)(m) = \sum_{i=0}^{n}a_{i}\frac{m^{i}}{i!}.
\]
The coefficient $a_{n}=d(F)$ is called the \emph{multiplicity} of $F$. The \emph{reduced Hilbert polynomial} of $F$ is defined by $p(F)(m)= P(F)(m)/d(F)$. A pure sheaf $F$\footnote{i.e., the dimension of the non-zero subsheaf of $F$ is the same as that of $F$} is \emph{stable} if for every nonzero proper subsheaf $F' \subset F$,
\[p(F')(m) < p(F)(m) \; \mathrm{for}\;\mathrm{all} \;m\gg 0.\]
Let $\cS(X,d)$ be the moduli space of isomorphism classes of stable sheaves with Hilbert polynomial $dm+1$.
\item \emph{Hilbert scheme}: Let $\cH(X,d)$ be the Hilbert scheme of
ideal sheaves $I_C$ of a curve $C$ in $X$ with Hilbert polynomial $\chi(\cO_C(m))=dm+1$.
\end{itemize}
The compactness of these moduli spaces is well known (\cite{FGetal05, FP97, HuLe10}) and their geometric properties (of small degree $d$) have been studied in various contents (\cite{EH82, PS85, FT04, Kie07, KM10, CK11, Che08b, CC11, CHK12}).

Let $R(X,d)$ be the quasi-projective variety parameterizing smooth rational curves of degree $d$ in $X$.
Let us denote by $\bM(X,d)$ (resp. $\bS(X,d)$, resp. $\bH(X,d)$) the closure of $R(X,d)$ in $\cM(X,d)$ (resp. $\cS(X,d)$,
resp. $\cH(X,d)$) and define it as the \emph{Kontsevich} (resp. \emph{Simpson}, resp. \emph{Hilbert}) \emph{compactification}.

Previously, the compactifications $\bM(\PP^r,d)$, $\bS(\PP^r,d)$, and $\bH(\PP^r,d)$ were compared using birational morphisms for the lower degree cases $d\leq 3$ (\cite{Kie07} and \cite{CK11}). They proved that these spaces are related by several blow-ups/-downs with geometric meaningful centers (See Section \ref{sec:2.1} for detail). The key technique employed for the comparison was to use the elementary modification of sheaves and the variation of geometric invariant theory (VGIT) quotients (\cite{HuLe10, Tha96}). Chung et al. subsequently generalized the aforementioned results \cite{CK11, Kie07} when the projective variety $X$ satisfies some suitable conditions (i.e., the condition (I)-(IV) in Section \ref{sec:2.1}).
For instance, all conditions are satisfied when $X$ is a projective homogeneous variety.
In this paper, we improve upon this by comparing compactified moduli spaces even when the del Pezzo $3$-fold $V_5$ does not satisfy the conditions.
\subsection{Results}
We study the case of the del Pezzo $3$-fold $V_5=\Gr(2,5)\cap \PP^6$, which is uniquely determined by a general linear section $\PP^6$ of the Grassmannian variety $\mathrm{Gr}(2,5)\subset \PP^9$ under the Pl\"ucker embedding. Many algebraic geometers have studied the Hilbert scheme of rational curves in $V_5$ of lower degrees from different perspectives (\cite{FN89, ILi94, LLSS17,San14, KPS18, CHL18}).
\begin{proposition}[\protect{\cite[Theorem 1]{FN89}, \cite[Proposition 1.2.2]{ILi94}, and \cite[Proposition 2.46]{San14}}]
The Hilbert scheme of rational curves of degree $d$ in $V_5$ is isomorphic to
\begin{equation*}
\cH(V_5,d)=
    \begin{cases*}
       \PP^2 & if $d=1$, \\
      \PP^4 & if $d=2$, \\
      \mathrm{Gr}(2,5) & if $d=3$.
    \end{cases*}
  \end{equation*}
\end{proposition}
In this study, we compare three moduli spaces $\cM(V_5,3)$, $\cS(V_5,3)$, and $\cH(V_5,3)$ for the case $d=3$ (i.e., twisted cubic curves).
Because the defining equation of $V_5$ is generated by degree $\leq 2$, it can be verified that $\cH(V_5,3)\cong\cS(V_5,3)$ and thus the Simpson space $\cS(V_5,3)$ is irreducible (Proposition \ref{hilbsimp}).

On the other hand, the space $\cM(V_5,3)$ proved not to be irreducible (Proposition \ref{mainprop2} compared with the general result \cite[Theorem 7.9]{LT17}). Two irreducible components of $\cM(V_5,3)$ are known to exist: the closure $\bM(V_5,3)$ of smooth twisted cubic curves and the relative stable map space $\cM(\cZ,3)$. Here $\cZ$ is the universal scheme of the Hilbert scheme $\bH(V_5,1)$ of the lines in $V_5$. Furthermore, the intersection $\bM(V_5,3)$ and $\cM(\cZ,3)$ consists of the space of stable maps into \emph{non-free lines} in $V_5$ (Definition \ref{freevsnon}).

By definition, $\bM(V_5,3)$ and $\bS(V_5,3)$ are obviously birationally equivalent. However, because $V_5$ does not satisfy condition (I) of Section \ref{sec:2.1} (Remark \ref{extob}), the comparison result of \cite{CHK12} cannot be applied in the case in which the initial point of the comparison is Kontsevich's compactification. Instead of this, our comparison of moduli spaces starts at the Simpson compactification $\bS(V_5,3)$.
Let us define a birational map $\Psi^{\mathrm{I}}:\bS(V_5,3)\dashrightarrow \bM(V_5,3)$, which is the inverse correspondence of $[f:C\lr X]\mapsto  f_*\cO_C$ for a stable map $[f:C\lr X]\in \bM(V_5,3)$ (cf. \eqref{raoriginal}).
Let $\bar{\Theta}_1$ be the locus of structure sheaves of a triple line lying on a quadric cone in $V_5$ and let $\bar{\Theta}_2$ be the closure of the locus of stable sheaves supported on an ordered pair of lines in $V_5$ (Section \ref{sec:3.1} and Section \ref{sec:3.2}). It is shown that $\bar{\Theta}_1$ and $\bar{\Theta}_2$ are smooth and $\bar{\Theta}_1\subset\bar{\Theta}_2$ (Proposition \ref{triple} and Proposition \ref{blcenter2}). The main result is to obtain a \emph{partial} desingularized model of $\bM(V_5,3)$ by blowing up $\bS(V_5,3)$ with centers $\bar{\Theta}_i$, $i=1,2$. Specifically,
\begin{theorem}\label{mainthm}
Under the above notations,
let $$\Psi^{\mathrm{I}}:\bS(V_5,3)\dashrightarrow \bM(V_5,3)$$ be the birational map. Then,
\begin{enumerate}
\item the undefined locus of $\Psi^{\mathrm{I}}$ is $\mathrm{Bs}(\Psi^{\mathrm{I}})=\bar{\Theta}_2$.
\item The map $\Psi^{\mathrm{I}}$ extends to a birational regular morphism $\widetilde{\Psi}^{\mathrm{I}}$ by the two weighted blow-ups of $\bS(V_5,3)$ along $\bar{\Theta}_1$ followed by the strict transform of $\bar{\Theta}_2$.
\item The blown-up space $\bM_3(V_5,3)$ of $\bS(V_5,3)$ has at most finite group quotient singularity.
\end{enumerate}
\end{theorem}
\[\xymatrix{\bM_3(V_5,3)\ar[d]_{\mathrm{strict}\;\mathrm{trans.}\;\mathrm{of}\;\bar{\Theta}_2}\ar[ddr]^{\widetilde{\Psi}^{\mathrm{I}}}&\\
\bM_4(V_5,3)\ar[d]_{\bar{\Theta}_1}&\\
\bS(V_5,3)\ar@{-->}[r]^{\Psi^{\mathrm{I}}}&\bM(V_5,3).}\]

The main ingredient of the proof is to use Theorem \ref{th1.3} when $X$ is a Grassmannian variety $\Gr(2,5)$. That is, we prove that the restriction of the blow-up maps \eqref{blowup2} over $\bS(V_5,3)$ coincides with the blowing up maps in item (2) of Theorem \ref{mainthm}. To do this, it is sufficient to check that the set-theoretic intersection of blow up centers is the scheme-theoretic one. A sufficient condition under which this is valid was discovered by Li (\cite[Lemma 5.1]{Li09}). Let $A$ and $B$ be smooth closed subvarieties of a smooth variety and let the set-theoretic intersection $U=A\cap B$ be smooth. If
\begin{equation}\label{cleanlycondition}
T_u U= T_u A \cap T_u B
\end{equation}
for each $u\in U$, then $U$ is the scheme-theoretic intersection (i.e., $I_A + I_B =I_U$). In this case we say that $A$ and $B$ \emph{cleanly} intersect along the subvariety $U$.

A careful analysis of the blow up maps of item (2) in Theorem \ref{mainthm}, reveals that the extended birational morphism $ \widetilde{\Psi}^{\mathrm{I}}$ is a \emph{small} map (See Section \ref{sec:5.2} for the definition).
From this and item (3) of Theorem \ref{mainthm}, the intersection Poincar\'e polynomial of $\bM (V_5,3)$ (Corollary \ref{corpoin}) is obtained.
\begin{remark}
For the case $d=2$, one can prove that $\cM(V_5,2)$ also consists of two irreducible components and the smooth blow up of $\cH(V_5,2)(=\PP^4)$ along the locus of non-free lines (Lemma \ref{nonfreeinconic}) is isomorphic to $\bM(V_5,2)$ by an argument parallel to that of Kiem (\cite{Kie07}).
\end{remark}
\subsection{Contents of the paper}
In Section \ref{sec:preliminaries}, we review the results of our previous study in \cite{CHK12} and collect some interesting properties of rational curves in $V_5$. Stable sheaves supported on non-reduced curves are presented in Section \ref{sect:3}. It is proven that stable sheaf $F$ on $V_5$ depends only on its multiplicity for each irreducible component of $\mathrm{Supp}(F)$ (Lemma \ref{typeofdeg1} and Lemma \ref{typeofdeg}). Furthermore, we describe explicitly the parameter space of stable sheaves having a non-reduced support via a local computation (Proposition \ref{triple} and Proposition \ref{blcenter2}). In Section \ref{sec:4}, we explain the global geometry of the Kontsevich space $\cM(V_5,3)$. We describe the irreducible components $\cM(V_5,3)$ and its intersection part in Proposition \ref{mainprop2}. In the last section, we prove Theorem \ref{mainthm} and calculate the intersection Poincar\'e polynomial of $\bM(V_5,3)$ (Corollary \ref{corpoin}).
\begin{notation}
\begin{itemize}
\item Let us denote by $\Gr(k,n)$ the Grassimannian variety parameterizing $k$-dimensional subspaces in a fixed vector space $V$ with $\dim V=n$.
\item Let $\PP_{(w_0,w_1,...,w_r)}^r$ be the weighted projective space of dimension $r$ with weight $(w_0,w_1,...,w_r)$.
\item We sometimes do not distinguish the moduli point $[x]\in \cM$ and the object $x$ parameterized by $[x]$ when no confusion can arise.
\item All of the exact sequences and the extension groups are considered in the ambient space but not in the support of relevant sheaves.
\end{itemize}
\end{notation}

\section{Preliminaries}\label{sec:preliminaries}
In this section, we recall the results of our comparison between the Kontsevich and Simpson space, which we intensively studied previously \cite{CHK12}. In addition, we discuss some algebro-geometric properties of lines and conics in $V_5$. We finally mention the well-known fact of the deformation theory of maps and sheaves. Hereinafter, the abbreviated notations $\bM(V_5)$ (or $\bM$) are at times used instead of $\bM(V_5,d)$ etc. when the meaning is clear from the context.
\subsection{Summary of the result in \cite{CHK12}}\label{sec:2.1}
In \cite{CHK12}, as a generalization of the case $X=\PP^r$ (\cite{CK11}), the authors compared the compactifications of rational curves when a smooth projective variety $X$ satisfies the following conditions (\cite[Lemma 2.1]{CHK12}).
\begin{enumerate}[(I)]
\item $\rH^1(\PP^1,f^*T_X)=0$ for any morphism $f:\PP^1\to X$.
\item $\mathrm{ev}:\cM_{0,1}(X,1)\to X$ is \emph{smooth} where
$$\cM_{0,1}(X,1)=\{(f:\PP^1\to X, p\in
\PP^1)\,|\,\mathrm{deg} f^*\cO_X(1)=1 \}$$ is the moduli space of
$1$-pointed lines on $X$ and $\mathrm{ev}$ is the evaluation map at the marked point.
\item The moduli space $F_2(X)$ of the planes in $X$ is smooth.
\item The defining ideal $I_X$ of $X$ in $\PP^r$ is generated by quadratic polynomials.
\end{enumerate}
For example, one can easily check that the Grassmannian variety $\Gr(k,n)$ satisfies all of conditions (I)-(IV).
For the various $X$ satisfying these conditions, it was previously proved that compactifications of rational curves of degree $\leq 3$ are related by explicit blow-ups/downs (\cite{CHK12}).
We recall the detail of the case $d=3$ for later use.
By taking the direct image $f_*\cO_C$ for a stable map $f:C\to X$ in $\bM=\bM (X,3)$, we obtain a birational map
\begin{equation}\label{raoriginal}
\Psi:\bM\dashrightarrow \bS=\bS(X,3),\; [f:C\lr X]\mapsto f_*\cO_C.
\end{equation}
The undefined locus of $\Psi$ (i.e., the locus of unstable sheaves) is the union of two subvarieties;
\begin{enumerate}
\item the locus $\Gamma_0^1$ of stable maps $f:C\lr X$ such that the image $f(C)=L$ is a line,
\item the locus $\Gamma_0^2$ of stable maps $f:C\lr X$ such that $f(C)=L_1\cup L_2$ is a pair of lines.\end{enumerate}

If $f\in \Gamma_0^1$, then $f_*\cO_C=\cO_L\oplus \cO_L(-1)^2$ and the normal space of $\Gamma^1_0$ in $\bM$ at $f$ is
\begin{equation}\label{ex1}
\Hom (\CC^2,\Ext^1_X(\cO_L,\cO_L(-1))).
\end{equation}
If $f\in \Gamma_0^2$ such that the restriction map $f|_{f^{-1}(L_i)}$ is an $i$-fold covering map onto the image $L_i$, then there exists a short exact sequence $\ses{\cO_{L_1\cup L_2}}{f_*\cO_C}{\cO_{L_2}(-1)}$. The normal space of $\Gamma^2_0$ in $\bM$ at $f$ is isomorphic to
\begin{equation}\label{ex2}
\Ext_X^1(\cO_{L_1\cup L_2}, \cO_{L_2}(-1)).
\end{equation}

Let $\bM_1$ be the blow-up of $\bM$ along $\Gamma_0^1$. By taking the elementary modification of sheaves with respect to the quotient $f_*\cO_C\twoheadrightarrow \cO_L(-1)^2$, we have an extension map
$
\bM_1\dashrightarrow \bS
$
of the birational map $\Psi$ in \eqref{raoriginal} where the locus of unstable sheaves in $\bM_1$ consists of two subvarieties;
\begin{enumerate}
\item the strict transform $\Gamma_1^2$ of $\Gamma_0^2$,
\item the subvariety $\Gamma_1^3$ of the exceptional divisor $\Gamma_1^1$, which are fiber bundles over $\Gamma^1_0$ with fibers
$$\PP \Hom_1 (\CC^2,\Ext^1_X(\cO_L,\cO_L(-1)))\cong \PP^1\times \PP \Ext^1_X(\cO_L,\cO_L(-1)),$$
where $\Hom_1$ denotes the locus of rank $1$ homomorphisms.
\end{enumerate}
By blowing up $\bM_1$ along $\Gamma_1^2$ followed by $\Gamma_2^3$ (the strict transform of $\Gamma_1^3$ along the second blow-up map)
and by again conducting an elementary modification of the sheaves, we obtain a birational morphism $\Psi_3:\bM_3\lr \bS$
which extends the original birational map $\Psi$ in \eqref{raoriginal}.
\begin{equation}\label{blowup1}
\xymatrix{\bM_3\ar[d]_{\Gamma_2^3}\ar[rrdd]^{\Psi_3}&&\\
\vdots\ar[d]_{\Gamma_0^1}\ar@{-->}[rrd]&&\\
\bM\ar@{-->}[rr]^{\Psi}&&\bS
}\end{equation}
A study of the analytic neighborhoods of $\Gamma^1_0$ and $\Gamma^2_0$ by using blow-up maps reveals that the local structure of the map $\Psi_3$ is
completely determined by a VGIT-quotients (see \cite[Section 4.4]{CK11} for a detail description). Eventually, we have a sequence of blow-down maps
\begin{equation}\label{blowup2}
\Psi_3:\bM_3\stackrel{\Phi_3}{\longrightarrow} \bM_4\stackrel{\Phi_4}{\longrightarrow} \bM_5\stackrel{\Phi_5}{\longrightarrow} \bM_6=\bS
\end{equation}
such that $\Psi_3=\Phi_5\circ \Phi_4\circ\Phi_3$. As we study the parameterized sheaves by the exceptional divisors while blowing up \eqref{blowup1}, the blow up centers of $\Phi_{5-j}$ in \eqref{blowup2} are described in terms of stable sheaves ($j=0,1,2$). Let $\Theta_j$ be the blow-up center of the birational map $\Phi_{5-j}:\bM_{5-j}\lr \bM_{6-j}$ in \eqref{blowup2}.
\begin{lemma}\label{bcenter}
Let $L$ and $L'$ be lines ($L \cap L'=\{\mathrm{pt}.\}$) in a smooth projective variety $X$ satisfying the conditions: (I)-(IV). Then,
\begin{enumerate}
\item The locus $\Theta_0$ parameterizes stable sheaves $F$ such that they fits into the non-split extension
\[
\ses{\cO_L(-1)\oplus \cO_L(-1)}{F}{\cO_L}.
\]
Thus, $\Theta_0$ is a fibration over $\bH(X,1)$ with fiber
$$\PP(\Ext_X^1(\cO_L,\cO_L(-1)\oplus \cO_L(-1)))^{\mathrm{s}}\git SL(2) \cong \mathrm{Gr}(2, \mathrm{dim} \Ext_X^1(\cO_L, \cO_L(-1))),$$
where $(\cdot)^{\mathrm{s}}$ denotes the locus of stable points (or equivalently, stable sheaves).
\item The locus $\Phi_5(\Theta_1)$ parameterizes sheaves $F$ such that they fit into the non-split extension
\[
\ses{\cO_{L}(-1)}{F}{\cO_{L^2}},
\]
where $L^2$ is a planar double line (\cite[page 41, XI]{EH82}).
Hence, it is a fibration over $\bH(X,1)$ with fiber
$$\PP( \Ext_X^1(\cO_{L^2},\cO_L(-1))).$$
\item
The locus $\Phi_5\circ\Phi_4(\Theta_2)$ is the closure of the locus that parameterizes sheaves $F$ such that they fit into the non-split extension
\[
\ses{\cO_{L}(-1)}{F}{\cO_{L\cup L'}}.
\]
Therefore, it is the closure of a fibration over the locus of intersecting lines in $\bH(X,1)\times \bH(X,1)\setminus \Delta$ with fiber $$\PP( \Ext_X^1(\cO_{L\cup L'},\cO_L(-1)))$$
where $\Delta$ is the diagonal of $\bH(X,1)\times \bH(X,1)$.
\end{enumerate}
\end{lemma}
\begin{proof}
We refer the reader to \cite[Section 4.1]{CHK12} for the proof of the claims.
\end{proof}
We can summarize the above discussion as follows.
\begin{theorem}\tcite{Theorem 1.7}{CHK12}\label{th1.3}
Let $X\subset \PP^r$ be a smooth projective variety satisfying conditions (I)-(IV) above.
$\bS(X,3)$ is obtained from $\bM(X,3)$ by blowing up along $\Gamma^1_0$,
$\Gamma^2_1$, and $\Gamma^3_2$
 and then
blowing down along $\Theta_1$, $\Theta_2$, and $\Theta_3$ (cf. \eqref{blowup1} and \eqref{blowup2}).
\end{theorem}
Sometimes we denote $\Theta_i$ (resp. $\Phi_i$) by $\Theta_i(X)$ (resp. $\Phi_i^X$) for stressing the ambient variety $X$.
In Theorem \ref{th1.3}, the blow-up map $\Phi_j^X$ for a projective variety $X\subset \PP^r$ is simply the restriction map $\Phi_j^X=\Phi_j^{\PP^r}|_X$ because the space of curves in $X$ has a cleanly intersection with the exceptional center in the case $\PP^r$. Thus, the fiber of the exceptional divisor of $\Phi_i^X$ is that of $\Phi_i^{\PP^r}$. For instance, the exceptional divisor of $\Phi_3^X$ (resp. $\Phi_4^X$) in \eqref{blowup2} is a $\PP_{(1,2,2)}^2$ (resp. $\PP_{(1,2,2,3,3)}^4$)-fibration over its base $\Theta_2(X)$ (resp. $\Theta_1(X)$) (\cite[Section 4.4]{CK11}). The main goal of this study is to prove that the same phenomenon occurs even though the variety $X=V_5$ does not satisfy condition (I).

\begin{remark}\label{triline}
The stable sheaf $F$ in item (2) of Lemma \ref{bcenter} is generically of the form $F\cong \cO_{L^3}$ where $L^3$ is defined by \emph{the triple line lying on a quadratic cone} (\cite[page 41, XIV]{EH82} and \cite[Example 4.16]{CK11}).
\end{remark}
\subsection{Lines and conics in $V_5$}
Some algebro-geometric properties of lines and (non-reduced) conics in $V_5$ are required to describe the blow-up centers of $\bS(V_5,3)$.
\begin{proposition}[\protect{\cite[Section 1]{FN89}}]\label{linesinv1}
The normal bundle of a line $L$ in $V_5$ is isomorphic to $$N_{L/V_5}\cong \cO_L(1)\oplus \cO_L(-1)\;\mbox{ or }\; \cO_L\oplus \cO_L.$$
\end{proposition}
\begin{definition}\label{freevsnon}
The line of the first (resp. second) type in Proposition \ref{linesinv1} is defined as a \emph{non-free} (resp. \emph{free}) line.
\end{definition}
The space of the non-free lines in $V_5$ provides some interesting subvarieties in the Hilbert schemes of higher degree rational curves.
\begin{lemma}[\protect{\cite[Section 2]{FN89}}]\label{nonfreeinconic}
The locus of non-free lines is a smooth conic in the Hilbert scheme $\bH(V_5,1)=\PP^2$.
\end{lemma}
Let us define the \emph{double line} as the non-split extension (stable) sheaf $F$
\[
\ses{\cO_L(-1)}{F}{\cO_L},
\]
where $L$ is a line.
In fact, the sheaf $F$ is isomorphic to $F\cong\cO_{L^2}$, where $L^2$ is a non-reduced plane conic (i.e., the reduced support of $F$ is $\mathrm{red}\mathrm{Supp}(F)=L$ and $\chi(F(m))=2m+1$). From $\Ext^1(\cO_L, \cO_L(-1))\cong \rH^0(N_{L/V_5}(-1))$, the line $L$ of the double line $\cO_{L^2}$ must be non-free by Proposition \ref{linesinv1}. A geometric description of double lines in $\bH(V_5,2)$ was provided by Iliev (\cite{ILi94}).
\begin{proposition}[\protect{\cite[Proposition 1.2.2]{ILi94}}]\label{double}\label{plane}
The locus of the double lines is a smooth, rational quartic curve in the Hilbert scheme $\bH(V_5,2)=\PP^4$.
\end{proposition}
A description of the normal bundle of a conic in $V_5$ is used several times in subsequent sections.
\begin{proposition}[\protect{\cite[Proposition 2.32]{San14}}]\label{linesinv}
Let $\cU|_{V_5}$ be the restricted bundle on $V_5$ of the universal rank two sub-bundle $\cU$ on $\mathrm{Gr}(2,5)$.
The ideal sheaf $I_C$ of $[C]\in \bH(V_5,2)$ has a locally free resolution
\begin{equation}\label{eq3}
\ses{\cO_{V_5}(-1)}{\cU|_{V_5}}{I_C}.
\end{equation}
Especially, the normal bundle of the conic $C$ in $V_5$ is isomorphic to
$$
N_{C/V_5}\cong \cU^*|_C.
$$
\end{proposition}
\subsection{Deformation theory of stable maps and sheaves}
The deformation theories of maps and sheaves are used for the analysis of the intersection part of the blow-up centers.
For the reader’s convenience, we address well-known facts about the deformation theory.
\begin{proposition}[\protect{\cite[Proposition 1.4, 1.5]{LT98}}]
Let $f:C\lr X$ be a stable map to a smooth projective variety $X$. Then the tangent space (resp. the obstruction space) of $\cM(X,d)$ at $[f]$ is given by
\begin{equation}\label{2012}
\Ext^1_C([f^*\Omega_X\to\Omega_C],\cO_C)\; (\mbox{resp. }\Ext^2_C([f^*\Omega_X\to\Omega_C],\cO_C)),
\end{equation}
where $[f^*\Omega_X\to\Omega_C]$ is a complex of sheaves concentrated at the degrees $-1$ and $0$.
\end{proposition}

\begin{lemma}\label{spectralseqmap}
Let $Y$ be a locally complete intersection of a smooth projective variety
$X$. Let $f:C \lr Y \subset X$ be a stable map that factors through
$Y$. Then there exists an exact sequence:
\[\begin{split}
0 \lr& \Ext^1 ([f^*\Omega_Y \lr \Omega_C] ,\cO_C) \lr \Ext^1
([f^*\Omega_X \lr \Omega_C] ,\cO_C) \lr \rH^0(f^* N_{Y/X})
\\
\lr& \Ext^2 ([f^*\Omega_Y \lr \Omega_C] ,\cO_C) \lr \Ext^2
([f^*\Omega_X \lr \Omega_C] ,\cO_C)\lr \rH^1(f^* N_{Y/X})\lr0,
\end{split}\]
where $N_{Y/X}$ is the normal bundle of $Y$ in $X$.
\end{lemma}
\begin{proof}
By applying the octahedron axiom (\cite[Proposition 1.4.4]{KS90}) for the derived category of sheaf
complexes on $C$ to the composition
$$f^*\Omega_{X}\lr f^*\Omega_{Y}\lr  \Omega_C,$$
we obtain a distinguished triangle
\[
f^*N_{Y/X}^*[1]\lr [f^*\Omega_{X}\lr \Omega_C]\lr
[f^*\Omega_{Y}\to \Omega_C] \mapright{[1]}.
\]
By taking the functor $\Hom(-,\cO_C)$ in this sequence, we obtain the result.
\end{proof}
\begin{proposition}[\protect{\cite[Proposition 2.A.11]{HuLe10}}] \label{defsheaves}
Let $F$ be a stable sheaf on a smooth projective variety $X$. Then the tangent space (resp. the obstruction space) of $\cS (X,d)$ at $[F]$ is given by
\begin{equation}
\Ext^1_X(F,F)\; (\mbox{resp. }\Ext^2_X(F,F)).
\end{equation}
\end{proposition}

\begin{lemma}
Let $Y \stackrel{i}{\subset} X$ be a smooth, closed subvariety of the smooth variety $X$.
If $F$ and $G \in \mathrm{Coh}(Y)$, then there is an exact sequence
\begin{equation}\label{thomas2}
\begin{split}
0 \to \Ext^1_{\cO_Y}(F, G) \to \Ext^1_{\cO_X}(i_*F, i_*G) &\to \Hom_{\cO_Y}(F,G\otimes N_{Y/X}) \\
&\to \Ext^2_{\cO_Y}(F, G) \to \Ext^2_{\cO_X}(i_*F, i_*G).
\end{split}
\end{equation}
\end{lemma}
\begin{proof}
This is the base change spectral sequence in \cite[Theorem 12.1]{mcc01}.
\end{proof}
\section{Non-reduced cubic curves in $\bS(V_5,3)$}\label{sect:3}
From now on, we will often write $\cH(X)$ or $\bH(X)$ (resp. $\cS(X)$ or $\bS(X)$) instead of $\cH(X,d)$ or $\bH(X,d)$ (resp. $\cS(X,d)$ or $\bS(X,d)$) when the meaning is clear from the context.
\begin{proposition}\label{hilbsimp}
The canonical correspondence
$$\bH(V_5)\rightarrow \bS(V_5),\; I_C\mapsto \cO_C$$
is an isomorphism. Hence, $\bS(V_5)=\mathrm{Gr}(2,5)$.
\end{proposition}
\begin{proof}
The entire Hilbert scheme $\cH(V_5)$ is the Hilbert compactification $\cH(V_5)=\bH(V_5)(\cong \Gr(2,5))$ by \cite[Corollary 1.39, Proposition 2.46]{San14} which parameterizes the ideal sheaf of Cohen-Macaulay (CM) curves with Hilbert polynomial $3m+1$.
On the other hand, the Simpson space $\cS(\PP^r)$ consists of two irreducible components: the space of structure sheaves of CM-curves and the space of non-split extensions of a planar cubic curve by a point on the curve (\cite{FT04} and \cite[Lemma 2.1]{CK11}). Let $F$ be a stable sheaf parameterized by $\cS(V_5)$ such that $F$ is supported on a planar cubic curve $C$. Then the curve $C$ is contained in $V_5$ only if the linear spanning $\langle C\rangle$ of $C$ is contained in $V_5$ (cf. \cite[Lemma 4.14]{CHK12}). Since $V_5$ is a Fano threefold, if it includes a plane, it can be contracted by an extremal contraction (which should be divisorial) (\cite[Theorem 2.5]{BSW90}). But this is not possible as $V_5$ is of Picard rank one. Therefore, the space $\cS(V_5)$ parameterizes structure sheaves of CM-curves and thus the canonical correspondence is an isomorphism. 
\end{proof}
In this section, we describe the stable sheaves parameterized by the complement $\bS(V_5)\setminus R(V_5)$; especially, we focus on the non-reduced curves.
\subsection{Triple lines in $V_5$}\label{sec:3.1}
\begin{lemma}\label{typeofdeg1}
Let $F\in \bS(V_5)$ be a stable sheaf supported on a line $\mathrm{red}\mathrm{Supp}(F)=L$ in $V_5$. Then,
\begin{enumerate}
\item $L$ is non-free and
\item $F$ is isomorphic to $$F\cong \cO_{L^3}$$ the structure sheaf of the unique triple line $L^3$ lying on a quadric cone.
\end{enumerate}
\end{lemma}
\begin{proof}
The stable sheaf $F$ is isomorphic to a structured sheaf supported on a line by Proposition \ref{hilbsimp}. The list published by Eisenbud and Harris \tcite{page 40}{EH82} shows that there are two possible cases: item (1) and (2) of Lemma \ref{bcenter}.
However, $$\text{dim}\Ext^1(\cO_L,\cO_L(-1))=\mathrm{dim} \rH^0(N_{L/V_5}(-1))\leq1$$ for any line $L$ in $V_5$, the GIT-quotient in item (1) of Lemma \ref{bcenter} is an empty set. Therefore, $F\cong\cO_{L^3}$, where the triple line $L^3$ lies on a quadric cone (item (2) of Lemma \ref{bcenter}). In this case, the sheaf $F$ fits into the short
exact sequence
\[\ses{\cO_{L}(-1)}{F=\cO_{L^3}}{\cO_{L^2}},\]
where $L^2$ is the planar double line. Note that the double line $L^2$ uniquely exists if and only if $L$ is non-free.
This comes from the equality $\Ext^1(\cO_L, \cO_L(-1))\cong\rH^0(N_{L/V_5}(-1))$ and the result of Proposition \ref{linesinv1}. Let us show that $\Ext^1(\cO_{L^2},\cO_L(-1))=\CC$, which completes the proof of our claim.
From the structure sequence $\ses{I_{L^2}}{\cO_{V_5}}{\cO_{L^2}}$,
\[
\Ext^1(\cO_{L^2},\cO_L(-1))\cong \Ext^0(I_{L^2},\cO_L(-1)).
\]
We claim that $\Ext^0(I_{L^2},\cO_L(-1))=\CC$. From the resolution of the ideal sheaf $I_{L^2}$ in Proposition \ref{linesinv}, we obtain an exact sequence:
\[
0\lr \Ext^0(I_{L^2},\cO_L(-1))\lr \Ext^0(\cU,\cO_L(-1))\lr \Ext^0(\cO(-1),\cO_L(-1))\cong H^0(\cO_L)=\CC,
\]
where the middle term is $\Ext^0(\cU,\cO_L(-1))\cong \rH^0(\cU^*\otimes \cO_L(-1))$. Because the bundle $\cU$ is of rank two with $c_1(\cU)=-1$ and $L$ is a line in $\mathrm{Gr}(2,5)$, $\cU^*\otimes \cO_L(-1)=\cU|_L\cong\cO_L(-1)\oplus \cO_L$. Hence, $\text{dim}\Ext^0(I_{L^2},\cO_L(-1))\leq 1$. By definition, $L^2$ is planar and thus we have a short exact sequence $\ses{\cO_L(-1)}{\cO_{L^2}}{\cO_L}$. Taking the functor $\Ext(I_{L^2},-)$ in this exact sequence, we have
\[
0\lr \Ext^0(I_{L^2},\cO_L(-1))\lr \Ext^0(I_{L^2},\cO_{L^2}) \lr \Ext^0(I_{L^2},\cO_{L})\lr\cdots.
\]
The middle term $\Ext^0(I_{L^2},\cO_{L^2})\cong\CC^4$ because it is isomorphic to the tangent space of $\bH_2(V_5)=\PP^4$ at $L^2$. By employing a similar calculation as above, we have $\text{dim}\Ext^0(I_{L^2},\cO_{L})\leq 3$, which implies that $$\text{dim}\Ext^0(I_{L^2},\cO_L(-1))\geq 1.$$ This completes the proof.
\end{proof}
\begin{proposition}\label{triple}
The locus of triple lines lying on a quadric cone is a degree six smooth rational curve in $\bS(V_5)=\mathrm{Gr}(2,5)\subset \PP^9$.
\end{proposition}
The proposition can be proven by an explicit calculation with the help of the computer program, Macaulay2 (\cite{M2}). Recall that the del Pezzo variety $V_5$ is defined by the linear section $\mathrm{Gr}(2,5)\cap \PP^6 \subset \PP^9$. Let $\{p_{ij}\}_{0\leq i<j\leq 4}$ be the Pl\"ucker coordinates of $\PP^9$. We use the linear section $\PP^6=H_1\cap H_2 \cap H_3$ defined by
\[I_{\PP^6}=\langle p_{12}-p_{03},\; p_{13}-p_{24},\; p_{14}-p_{02}\rangle.\]
Furthermore, let us denote $p_i:=p_{0i}$ for $i=1,2,3,4$. We find the universal family of $\bS(V_5)=\bH(V_5)=\mathrm{Gr}(2,5)$ around a triple line (lying on a quadric cone). Recall the correspondence $i:\mathrm{Gr}(2,5)\lr \bH(V_5)$ described in Remark 2.47 of \cite{San14}.
The map $i$ defines a correspondence between a line $L\subset \PP^4$ and the closed subscheme $ i([L])=\sigma_{2,0}(L)\cap \PP^6\subset V_5$, where $$\sigma_{2,0}(L):=\{[L']| L\cap L'\neq \emptyset,\; L'\subset \PP^4 \}$$ is the Schubert subvariety of $\Gr(2,5)$. Let us construct a flat family of twisted cubic curves in $V_5$ around the point
\begin{equation}\label{ori}[L_0]=\begin{bmatrix}
1&0&0&0&0\\
0&1&0&0&0
\end{bmatrix}\in \mathrm{Gr}(2,5)=\bH(V_5).\end{equation}
An affine chart of $\Gr(2,5)$ at the point $[L_0]$ is locally given by
\begin{equation}\label{sch}
\begin{bmatrix}
1&0&a_2&a_3&a_4\\
0&1&b_2&b_3&b_4
\end{bmatrix},\end{equation}
where $(a_2,a_3,a_4,b_2,b_3,b_4)\in \CC_{(a,b)}^6\subset \mathrm{Gr}(2,5)=\bH(V_5)$. Let $\{e_0,e_1,e_2,e_3, e_4\}$ be the standard coordinate vector of the space $\CC^5$, which gives the original projective space $\PP^4$. Let us define $L$ by
$$L=\mathrm{span}\langle w_0=e_0+a_2e_2+a_3e_3+a_4e_4,\; w_1=e_1+b_2e_2+b_3e_3+b_4e_4\rangle.$$
Then the line $L'=\mathrm{span}\langle v_0,v_1\rangle$ where
\[\begin{split}
v_0&=w_0+tw_1=e_0+te_1+(a_2+tb_2)e_2+(a_3+tb_3)e_3+(a_4+tb_4)e_4,\\
v_1&=w_1+k_2e_2+k_3e_3+k_4e_4=e_1+(b_2+k_2)e_2+(b_3+k_3)e_3+(b_4+k_4)e_4
\end{split}\]
such that $L\cap L'=\mathrm{span}\langle v_0\rangle$.
The computer program Macaulay2 (\cite{M2}) enables the scheme theoretic closure
$\sigma_{2,0}(L)\subset V_5\times \CC_{(a,b)}^6\times \CC_{(t,k_2,k_3,k_4)}$ to be computed.
After eliminating the variables $\{t,k_2,k_3,k_4\}$ by the computer program (\cite{M2}) again, we obtain a flat family $\cC$ of twisted cubic curves over $\CC_{(a,b)}^6$
\[
\xymatrix{
\cC\ar@{^{(}->}[r]\ar[dr]&V_5\times \CC_{(a,b)}^6\subset \PP^9\times\CC_{(a,b)}^6 \ar[d]\\
&\CC_{(a,b)}^6,
}
\]
where the defining equations of $\cC$ are given by the following nine equations:
\begin{enumerate}[(a)]
\item $p_{34}=(a_4b_3-a_3b_4)p_1+a_3p_4-a_4p_{3}+b_3p_{2}-b_4p_{24}$ ;
\item $p_{24}=(a_4b_2-a_2b_4)p_1+a_2p_4+(-a_4+b_2)p_{2}-b_4p_{3}$ ;
\item $p_{23}=(a_3b_2-a_2b_3)p_1+(a_2-b_3)p_{3}-a_3p_{2}+b_2p_{24}-b_3p_{3}$ ;
\item $p_{14}=p_2$, $p_{13}=p_{24}$, $p_{12}=p_{3}$;
\item $F_1(p,a_i,b_i)=F_2(p,a_i,b_i)=F_3(p,a_i,b_i)=0$, where
\[\begin{split}
F_1(p,a,b)=&p_3p_4-p_2^2-b_4p_1p_3+b_2p_1p_2-a_4p_1p_2+a_2p_1p_4+(a_4b_2-a_2b_4)p_1^2;\\
F_2(p,a,b)=&-p_2p_3+a_2p_4^2+a_3p_1p_4-a_4p_1p_3-a_4p_2p_4+b_2p_2p_4+b_3p_1p_2-b_4p_2^2+\\
&a_4b_2p_1p_4+a_4b_3p_1^2-a_2b_4p_1p_4-a_3b_4p_1^2;\\
F_3(p,a,b)=&p_3^2+b_4p_2p_3-b_3p_1p_3-b_2p_3p_4+a_4p_2^2-a_3p_1p_2-a_2p_2p_4+a_2p_1p_3+\\
&a_3b_2p_1^2-a_4b_2p_1p_2-a_2b_3p_1^2+a_2b_4p_1p_2.\\
\end{split}\]
\end{enumerate}
\begin{remark}
\begin{enumerate}
\item The relations $F_1=F_2=F_3=0$ in (e) can be written in terms of the net of quadrics:
\begin{equation}\label{net}
\mathrm{rank}\begin{bmatrix}
a_4p_2-a_3p_1-a_2p_4&-a_2p_1-p_3&a_4p_1+p_2\\
b_4p_2-b_3p_1-b_2p_4+p_3&-b_2p_1+p_2&b_4p_1-p_4
\end{bmatrix}\leq 1.
\end{equation}
\item The line $[L_0]$ in \eqref{ori} (i.e., the origin of $\CC_{(a,b)}^6$) represents the triple line in $V_5$ defined by the ideal
$$
I_{L^3}=\langle p_2p_3,p_3p_4-p_2^2,p_3^2,p_{12},p_{13},p_{14},p_{23},p_{24},p_{34}\rangle.
$$
\end{enumerate}
\end{remark}

\begin{proof}[Proof of Proposition \ref{triple}]
We claim that the locus of triple lines in the affine chart $\CC_{(a,b)}^6\subset\mathrm{Gr}(2,5)$ is parameterized by the degree six smooth rational curve
\[\begin{bmatrix}
1&0&-s^3&\frac{3}{4}s^4&\frac{3}{2}s^2\\
0&1&-\frac{3}{2}s^2&s^3&3s
\end{bmatrix}\in \mathrm{Gr}(2,5)=\bH(V_5)\]
for $s\in \CC$.
From the defining equations in (d) of the linear section, we have \begin{equation}\label{eqrel} a_4 + b_2 =a_2  +b _3=b_2^2  - b_3 b_4  + a_3=0.\end{equation}
Plugging these equations into the net \eqref{net} and performing the elementary operation, we have
\begin{equation}\label{eq4}
\text{rank}\begin{bmatrix}
-2b_2y-b_3w-b_4z&z&y\\
-z&y&w
\end{bmatrix}\leq1,\end{equation}
where $x=p_1$, $ y=p_2-b_2p_1$, $z=b_3p_1-p_3$, and $w=b_4p_1-p_4$.
Note that the triple line is uniquely determined by its supporting line (Lemma \ref{typeofdeg1}).
However, one can easily see that the support of cubic curves in \eqref{eq4} is a line in the projective space $\PP^3_{\{x,y,z,w\}}$
if and only if $b_2=-\frac{3}{2}s^2$, $b_3=s^3$, and $b_4=3s$ for $s\in \CC$. Combined with the equation \eqref{eqrel}, we obtain the result.
\end{proof}

\subsection{Pair of lines in $V_5$}\label{sec:3.2}
\begin{lemma}\label{pairofline}\label{typeofdeg}
Let $C=L_1\cup L_2$ be a reducible conic in $V_5$ such that $L_1\cap L_2=\{\mathrm{pt}.\}$.
Then there exists a unique non-split extension \begin{equation}\label{non2}\ses{\cO_{L_2}(-1)}{F}{\cO_{C}}\end{equation}
such that $F\cong \cO_{L_1\cup L_2^2}$, where $p_a(L_2^2)=0$ or $-1$. Furthermore, the line $L_2$ is non-free if and only if $p_a(L_2^2)=0$ (i.e., the double line is planar).
\end{lemma}
\begin{proof}
Replacing $L^2$ (resp. L) by $C$ (resp. $L_2$) in Lemma \ref{typeofdeg1} and repeating the same computation, we have \begin{equation}\label{eq13}\Ext^1(\cO_{C},\cO_{L_2}(-1))\cong \CC.\end{equation}
Compare with item (3) of Lemma \ref{bcenter}. Because each stable sheaf $F\in \bS_3(V_5)$ is isomorphic to a CM-curve, we have $F\cong\cO_{L_1\cup L_2^2}$ by \cite[page 39-41, VII, XI]{EH82}.

For the second part, let us assume that $L_2$ is non-free. Taking the functor $\Ext(-,\cO_{L_2}(-1))$ in the structure exact sequence $\ses{\cO_{L_1}(-1)}{\cO_{C}}{\cO_{L_2}}$ of the reduced conic $C=L_1\cup L_2$, we obtain
\[
0=\Ext^0(\cO_{L_1}(-1),\cO_{L_2}(-1))\lr \Ext^1(\cO_{L_2},\cO_{L_2}(-1))\stackrel{\phi}{\lr}\Ext^1(\cO_{L_1\cup L_2},\cO_{L_2}(-1)).
\]
Because $\Ext^1(\cO_{L_2},\cO_{L_2}(-1))\cong \Ext^0(I_{L_2},\cO_{L_2}(-1))=\rH^0(N_{L_2/V_5}(-1))=\CC$ and \eqref{eq13}, the pull-back map $\phi$ is an isomorphism. Hence, $F$ is completely determined by the pull-back of $\cO_{L_2^2}$ such that $L_2^2$ is a planar double line.

Conversely, if $p_a(L_2^2)=0$, then the sheaf $\cO_{L_2^2}$ fits into the non-split extension $$\ses{\cO_{L_2}(-1)}{\cO_{L_2^2}}{\cO_{L_2}}.$$  Hence, $\Ext^1(\cO_{L_2},\cO_{L_2}(-1))=\rH^0(N_{L_2/V_5}(-1))\neq0$, which implies that $L_2$ is non-free.
\end{proof}
From Lemma \ref{pairofline}, we can say that a stable sheaf $F\in \bS(V_5)$ with the reduced support $L_1\cup L_2$ is unique whenever the multiplicity of the restriction $F|_{L_i}$ of $F$ to $L_i$ is $r(F|_{L_i})=i$.
The locus of such sheaves can be geometrically described by using \cite[Section 2]{San14}.
\begin{proposition}\label{blcenter2}
In the notation of Lemma \ref{pairofline}, the closure of the locus of stable sheaves of the form $\cO_{L_1\cup L_2^2}$ is isomorphic to the space of the ordered pairs of intersecting lines in $V_5$.
\end{proposition}
\begin{proof}
Following the notation of Section 2 in \cite{San14}, let $IQ\subset \bH_1(V_5)\times \bH_1(V_5)$ be the space of the ordered pairs of intersecting lines in $V_5$. It was proved that $IQ$ is isomorphic to the full flag variety $\mathrm{Fl}(1,2;\CC^3)$ and $IQ\cap \Delta$ is isomorphic to the locus (i.e., conic) of non-free lines in $V_5$, where $\Delta$ is the diagonal of $\bH_1(V_5)\times \bH_1(V_5)$ (\cite[Proposition 2.26 and Proposition 2.27]{San14}). However, there exists a two-fold ramified covering map $IQ\lr \bH(V_5,2)$ onto the locus of reducible conics in $V_5$ (\cite[Proposition 2.44]{San14}). Using these ones, one can easily construct a flat family of extension sheaves in \eqref{non2} over $IQ$ (\cite[Example 2.1.12]{HuLe10}), which induces a closed embedding $IQ\stackrel{i}{\subset} \bS(V_5,3)$. Note that under this closed embedding, $i(IQ\cap \Delta)$ is isomorphic to the locus of triple lines in Proposition \ref{triple}.
\end{proof}
\begin{notation}\label{notimpor}
Let us denote by $\bar{\Theta}_1(V_5)$ (resp. $\bar{\Theta}_2(V_5)$) the smooth sublocus of $\bS(V_5)$ parameterizing the stable sheaves in Proposition \ref{triple} (resp. \ref{blcenter2}). Compare with Lemma \ref{bcenter}.
\end{notation}
\section{Irreducible components of the Kontsevich space}\label{sec:4}
Let us recall that $\cM(V_5):=\cM(V_5,3)$ is the moduli space of stable maps $f:C\lr V_5$ with $g(C)=0$ and $\deg(f^*\cO_X(1))=3$. The smooth rational curve component in $\cM(V_5)$ is denoted by $\bM(V_5)$. There may be extra connected components because the moduli space $\cM(V_5)$ is not a smooth stack (Remark \ref{extob}). The goal of this section is to prove the following proposition.
\begin{proposition}\label{mainprop2}
Let $\cZ$ be the universal subscheme of the Hilbert scheme $\bH_1(V_5)$ of lines in $V_5$.
\begin{enumerate}
\item The stable map space $\cM(V_5)$ consists of two irreducible components: $\bM(V_5)$ and the relative stable maps space $\cM(\cZ)$ over $\bH_1(V_5)$.
\item The intersection part $\bM(V_5)$ and $\cM(\cZ)$ is the relative stable map space over the locus of non-free lines in $V_5$.
\end{enumerate}
\end{proposition}
Lehmann and Tanimoto \cite[Theorem 7.9]{LT17} proved that item (1) of Proposition \ref{mainprop2} holds for any degree $d\geq2$. However, the proof we present below differs from theirs to the best of the author’s knowledge. Furthermore, item (2) of Proposition \ref{mainprop2} seems to be new.
\subsection{Obstruction of a stable map}
We prove that the obstruction space vanishes for stable maps of which the image is a pair of lines.
\begin{lemma}\label{ob2}
Let $C'=L_1 \cup L_2$ be the pair of lines in $V_5$. Then $$\rH^1(T_{V_5}|_{C'})=0.$$
\end{lemma}
\begin{proof}
One of the two lines $L_1$ and $L_2$ should be free, based on the following reasoning.
Recall that the non-free line consists of a smooth conic in $\bH_1(V_5)=\PP^2$ (Lemma \ref{plane}). From \cite[Corollary 1.2]{FN89}, for a point $q=[L_q]$ in the smooth conic, the lines in $V_5$ parameterized by points in the tangent line $T_q\PP^2$ at $q$ are all of the lines meeting with the non-free line $L_q$.
Hence, we may assume that $L_1$ is a free line. By tensoring the tangent bundle $T_{V_5}$ in the structure sequence $\ses{\cO_{L_1}(-1)}{\cO_{C'}}{\cO_{L_2}}$, we have a long exact sequence
$$
\lr \rH^1(T_{V_5}|_{L_1}(-1))\lr \rH^1(T_{V_5}|_{C'})\lr  \rH^1(T_{V_5}|_{L_2}) \lr.
$$
From Proposition \ref{linesinv1}, we have $\rH^1(T_{V_5}|_{L_1}(-1))=\rH^1(T_{V_5}|_{L_2})=0$, which completes the proof of the claim.
\end{proof}

\begin{lemma}\label{ob1}
Let $f:C\lr C'\subset V_5$ be the stable map with $C'=L_1\cup L_2$, $L_1\neq L_2$ such that the restriction map $f|_{f^{-1}(L_i)}$ is an $i$-fold covering map onto $L_i$. Then the obstruction space of $\cM(V_5)$ at $[f]$ is
$$\Ext_C^2([f^*\Omega_{V_5}\to\Omega_C],\cO_C)=0.$$
\end{lemma}
\begin{proof}
Let us write shortly $\mathrm{Ob}_f(X):=\Ext^2_C([f^*\Omega_X\to\Omega_C],\cO_C)$ for a variety $X$.
Suppose that $\rH^1(f^*N_{C'/V_5})=0$ instantly. If this is true, then we have a commutative diagram (cf. \cite[Lemma 4.10]{CK11})
\[
\xymatrix{\rH^0(N_{C'/V_5})\ar[r]\ar[d]&\Ext^1(\Omega_{C'},\cO_{C'})\cong\CC\ar[r]\ar[d]^{\cong}&\rH^1(T_{V_5}|_{C'})=0\ar[d]&\\
\rH^0(f^*N_{C'/V_5})\ar[r]&\mathrm{Ob}_f(C')\cong\CC\ar[r]&\mathrm{Ob}_f(V_5)\ar[r]&\rH^1(f^*N_{C'/V_5})=0,
}
\]
where $\rH^1(T_{V_5}|_{C'})=0$ by Lemma \ref{ob2}. By the commutativity of the middle of the above diagrams, we have $\mathrm{Ob}_f(V_5)=0$.

Let us show that $\rH^1(f^*N_{C'/V_5})=0$. Consider the exact sequence $$\ses{N_{C'/V_5}}{f^*N_{C'/V_5}\cong f_*\cO_C \otimes N_{C'/V_5}}{N_{C'/V_5}|_{L_2}(-1)}$$ obtained from the exact sequence $\ses{\cO_{C'}}{f_*\cO_C}{\cO_{L_2}(-1)}$ by tensoring the normal bundle $C'$ in $V_5$. Yet, one can easily see that $\rH^1(N_{C'/V_5})=0$ by Lemma \ref{ob2} again. In addition, $$N_{C'/V_5}|_{L_2}(-1)=\cU^*|_{L_2}(-1)= \cU|_{L_2}\cong \cO_{L_2}\oplus \cO_{L_2}(-1)$$ because of \eqref{eq3}, $\mathrm{rank}(\cU)=2$, and $c_1(\cU)=-1$. Hence, $\rH^1(N_{C'/V_5}|_{L_2}(-1))=0$, which completes the proof of our claim.
\end{proof}
\begin{proof}[Proof of Proposition \ref{mainprop2}]
For $[f]\in \cM_3(V_5)$, if the degree of the image $f(C)$ is $\deg f(C)=3$ , then $C\cong f(C)$ (after contracting the central component\footnote{i.e., an irreducible component of $C$ mapping to a point under $f$} of the domain curve). Hence, $[\cO_{f(C)}]\in \bS(V_5)$, which implies that $[f]\in \bM(V_5)$.

Let $\Gamma_0^2$ be the locus of stable map $f:C\lr V_5$ such that $\mathrm{deg}f(C)=2$. Note that if $\mathrm{deg}f(C)=2$, then $f(C)=L_1\cup L_2$ with the restriction on $f^{-1}(L_i)$ is a $i$-fold covering map of $L_i$, $i=1,2$.
Therefore $\Gamma_0^2$ is isomorphic to a $\bM(L_2, 2)=\PP_{(1,2,2)}^2$-bundle over the space $IQ\setminus (IQ\cap \Delta)$ of the ordered pair of lines (cf. Proposition \ref{blcenter2}). In special, $\Gamma_0^2$ is irreducible and $\dim\Gamma_0^2=5$. Furthermore, the moduli space $\cM(V_5)$ at each $[f]\in \Gamma_0^2$ has at most finite group quotient singularity (Lemma \ref{ob1}) and thus $\bM(V_5)\cup \Gamma_0^2$ is irreducible. That is, $\Gamma_0^2\subset \bM(V_5)$.

If $\mathrm{deg}f(C)=1$, then $f$ factors through a line in $V_5$ and thus it lies in the relative stable maps space $\cM(\cZ)$ over the Hilbert scheme $\bH_1(V_5)$ of lines in $V_5$. Thus, it is another irreducible component of $\cM(V_5,3)$ because $\dim \cM(\cZ)= \dim\bM(V_5)=6$.

Finally, we prove item (2). Let $p: \Gr(2,5)\subset \PP^9$ be the Pl\"ucker embedding. Then $V_5$ is defined as a zero locus of a section $s$ of the vector bundle $p^*\cO(1)^{\oplus 3}$ over $\Gr(2,5)$. Therefore the moduli space  $\cM(V_5,3)$ can be regraded as a zero locus of the induced section $\widetilde{s}$ of the vector bundle $\cV:=\pi_*\text{ev}^*(p^*\cO(1)^{\oplus 3})$ over the space $\cM(\Gr(2,5),3)$. Here the map $\pi: \cC\lr \cM(\Gr(2,5),3)$ is the universal family and $\text{ev}: \cC \lr \Gr(2,5)$ is the evaluation map.
Since the dimension of each irreducible component of $\cM(V_5,3)$ is the expected dimension $6$$(=\dim \cM(\Gr(2,5),3)-\mathrm{rank}\cV$),
the space $\cM(V_5,3)$ is a local complete intersection (cf. \cite[Section 2]{BK13}). Hence, the intersection part of two irreducible components of $\cM(V_5,3)$ must be purely of dimension $5$ by \cite[Theorem 3.4]{Har62}.
Let $[f:C\lr L\subset V_5]\in \bM(V_5)\cap \cM(\cZ)$ be the intersection point such that $L$ is a line. Because $[f:C\lr L]\in \bM(V_5)$, it is possible to construct one parameter family of smooth twisted cubics $f_t:\PP^1\times\CC\lr C_t\subset V_5$ such that $\text{lim}_{t\lr0}f_t=f$. If we regard this one as one parameter family of stable shaves $\cO_{C_t}$ over $t\neq0$, then its limit (as a stable sheaf) must be supported on a non-free line $L$ by Lemma \ref{typeofdeg1}. The stable reduction of the flat family of stable maps $f_t$ ensures that the image $f(\PP^1)=L$ of the limit map $f$ is non-free. Because the intersection part should be purely of dimension $5$, the only possibility is the one claimed above.
\end{proof}

\begin{remark}\label{extob}
For $[f]\in \Gamma_0^1$, let $f:C\lr L\subset V_5$ be a stable map such that $L$ is non-free. By Lemma \ref{spectralseqmap}, we have an exact sequence
\[
\Ext^2 ([f^*\Omega_L \lr \Omega_C] ,\cO_C) \lr \Ext^2
([f^*\Omega_{V_5} \lr \Omega_C] ,\cO_C)\lr \rH^1(f^*N_{L/V_5})\lr0.
\]
The first term is $ \Ext^2 ([f^*\Omega_L \lr \Omega_C] ,\cO_C)=0$ because of the convexity of $L(\cong\PP^1)$. Also $\rH^1(f^*N_{L/V_5})\cong\rH^1(f_*\cO_C\otimes N_{L/V_5})\cong\CC^2$ by the adjunction formula and Proposition \ref{linesinv1}. Hence the obstruction space of $\cM(V_5)$ at $[f]$ does not vanish. Furthermore, from the exact sequence
\[
\Ext^1(\Omega_C,\cO_C)\lr \rH^1(f^*T_{V_5})\lr  \Ext^2
([f^*\Omega_{V_5} \lr \Omega_C] ,\cO_C)\lr 0,
\]
we have $\rH^1(f^*T_{V_5})\neq 0$.
\end{remark}
\begin{question}
For all $d\geq2$, is the intersection part of $\bM(V_5,d)$ and the relative stable map space $\cM(\cZ,d)$ the relative stable map space over the locus of non-free lines in $V_5$ ?
\end{question}
\section{Comparison between $\bM(V_5,3)$ and $\bS(V_5,3)$}
In this section, we compare $\bM(V_5)$ and $\bS(V_5)$ by using explicit (weighted) blow-up maps. As a corollary, we obtain the intersection cohomology group of $\bM(V_5)$.
\subsection{The proof of Theorem \ref{mainthm}}
Theorem \ref{mainthm} is proven in the following way.
Let us fix a closed embedding $V_5 \subset \mathrm{Gr}(2,5):=G$. We check that the blow-up maps of the diagram in \eqref{blowup2} can be applied in our space $\bS(V_5)$ and thus we have a desingularized model of the stable map space $\bM(V_5)$. The first blow-up map $\Phi_5$ in \eqref{blowup2} restricted to $\bS(V_5)$ is an isomorphism by the proof of Lemma \ref{typeofdeg1}; hence, $\Phi_5^{-1}(\bS(V_5))= \bS(V_5)$.
Therefore, it is sufficient to check that the blow-centers $\bar{\Theta}_1(G)=\Phi_5\circ \Phi_4 (\Theta_1(G))$ (resp. $\bar{\Theta}_2(G)=\Phi_5(\Theta_2(G))$) in Theorem \ref{th1.3} \emph{cleanly} intersect with $\bS(V_5)$ along the locus $\bar{\Theta}_1(V_5)$ (resp. $\bar{\Theta}_2(V_5)$) (Notation \ref{notimpor}).
This implies that the blow-up map $$\Phi_4^{V_5}: \bM_4(V_5)=\mathrm{bl}_{\bar{\Theta}_1(V_5)}\bS(V_5)\lr \bS(V_5)$$ with the center $\bar{\Theta}_1(V_5)$ is nothing but the restriction map of $\Phi_4$ over $\Phi_4^{-1}(\bS(V_5))$ in \eqref{blowup2}. Note that the base change property of the weighted blows-up can be obtained from Lemma 3.1 in \cite{MM07}. Let $\bM_3(V_5)$ be the blowing-up of $\bM_4(V_5)$ along $\Theta_2(V_5)$, which is the strict transform of $\bar{\Theta}_2(V_5)$ by the first blow-up map $\Phi_4^{V_5}$. The blown-up space $\bM_3(V_5)$ is the restriction of $\bM_3(G)$ by the same reasoning as before. Eventually, we obtain a birational morphism $$\bM_3(V_5)=\mathrm{bl}_{\Theta_2(V_5)}\bM_4(V_5)\longrightarrow \bM(V_5),$$
which extends the birational map $\bS(V_5)\dashrightarrow \bM(V_5)$. The space $\bM_3(V_5)$ has at most finite group quotient singularity because it is a weighted blown-up space of a smooth variety.
\begin{proof}[Proof of Theorem \ref{mainthm}]
Checking the cleanly intersection is sufficient to prove that the normal spaces of the blow-up centers are restricted versions of that of the case $G=\mathrm{Gr}(2,5)$.
Because the space of the structured sheaves of a CM-curve is open in $\cS(G)$, the deformation theory of sheaves in Proposition \ref{defsheaves} can be applied in our setting.
Recall that $\bar{\Theta}_1(G)$ is the space of sheaves $F$ fitting into the short exact sequence:
\begin{equation}\label{stable1}
\ses{\cO_L(-1)}{F}{\cO_{L^2}},
\end{equation}
where $L^2$ is a planar double line. The space $\bar{\Theta}_1(G)$ is a $\PP(\Ext_G^1(\cO_{L^2},\cO_L(-1)))$-bundle over a $\PP(\Ext_G^1(\cO_L,\cO_L(-1)))$-bundle over $\bH_1(G)$, where the fiber $\PP(\Ext_G^1(\cO_L,\cO_L(-1))$ parameterizes the choice of the plane containing line $L$. Let us show that there exists a canonical isomorphism of normal bundles:
\[
N_{\bar{\Theta}_1(V_5)/\bS(V_5),F}\cong N_{\bar{\Theta}_1(G)/\bS(G),F}
\]
for the stable sheaf $F\in \bS(V_5)$ of the type \eqref{stable1}. From the commutativity of the diagram
\begin{equation}\label{com1}
\xymatrix{0\ar[r]&T_F\bar{\Theta}_1(V_5)\ar[r]\ar[d]^{\psi_1}&T_F\bS(V_5)\ar[r]\ar[d]^{\psi_2}&N_{\bar{\Theta}_1(V_5)/\bS(V_5),F}\ar[r]\ar[d]&0\\
0\ar[r]&T_F\bar{\Theta}_1(G)\ar[r]&T_F\bS(G)\ar[r]&N_{\bar{\Theta}_1(G)/\bS(G),F}\ar[r]&0}
\end{equation}
it suffices to show that the induced map \begin{equation}\label{eq11}\mathrm{Coker}(\psi_1)\stackrel{\cong}{\longrightarrow} \mathrm{Coker}(\psi_2)\end{equation} is an isomorphism.

We first show that there exists a (non-canonical) isomorphism
\begin{equation}\label{iso1}
\mathrm{Coker}(\psi_1)\cong\Hom(\cO_{L^2}, \cO_L(-1)\otimes N_{V_5/G})\oplus \Hom(\cO_{L^2},\cO_{L^2}\otimes N_{V_5/G}),
\end{equation}
where the normal bundle of $V_5$ in $G$ is $N_{V_5/G}=\cO_{V_5}(1)^{\oplus3}$. Note that $\bar{\Theta}_1(V_5)$ (resp. $\bar{\Theta}_1(G)$) is a $\PP(\Ext_{V_5}^1(\cO_{L^2},\cO_L(-1)))$ (resp. $\PP(\Ext_{G}^1(\cO_{L^2},\cO_L(-1))$)-bundle over the locus $\bar{\Theta}_1^2(V_5)$ (resp. $\bar{\Theta}_1^2(G)$) of double lines in $V_5$ (resp. $G$) as viewed in the Hilbert scheme of conics (cf. Proposition \ref{plane}). In this setting, the first term in \eqref{iso1} is isomorphic to the normal space of the fiber of $\bar{\Theta}_1(V_5)$ and $\bar{\Theta}_1(G)$. The second term in \eqref{iso1} is isomorphic to the normal space of base spaces $$N_{\bar{\Theta}_1^2(V_5)/\bar{\Theta}_1^2(G),F}\cong\Hom(\cO_{L^2},\cO_{L^2}\otimes N_{V_5/G}),$$which can be checked by rewriting the diagram \eqref{com1} about the spaces $\bH(V_5,2)$ and $\bH(G,2)$.

Second, there is an isomorphism
\begin{equation}\label{eq12}\mathrm{Coker}(\psi_2)\cong\Hom(F,F\otimes N_{V_5/G})\end{equation}
by identifications $T_F\bS(V_5)=\Ext_{V_5}^1(F,F)$, $T_F\bS(G)=\Ext_{G}^1(F,F)$ (with dimension $6$ and $18$, respectively), and Lemma \ref{thomas2}.
However, the latter space in \eqref{eq12} fits into the middle term of the diagram
\[
\xymatrix{0\ar[r]&\Hom(F,\cO_L(-1)\otimes N_{V_5/G})\ar[r]&\Hom(F,F\otimes N_{V_5/G})\ar[r]&\Hom(F,\cO_{L^2}\otimes N_{V_5/G})\ar[r]&0\\
&\Hom(\cO_{L^2},\cO_L(-1)\otimes N_{V_5/G})\ar^{\cong}[u]&&\Hom(\cO_{L^2},\cO_{L^2}\otimes N_{V_5/G})\ar[u]^{\cong}&
},
\]
which has its origins in \eqref{stable1}. Hence, the isomorphism in \eqref{eq11} holds.

The use of a similar computation as before enables a canonical isomorphism
\[
N_{\bar{\Theta}_2(V_5)/\bS(V_5),F}\cong N_{\bar{\Theta}_2(G)/\bS(G),F}
\]
to be obtained for the stable sheaf $F\in \bar{\Theta}_2(V_5)$ fitting into the non-split extension $\ses{\cO_{L_2}(-1)}{F}{\cO_{L_1\cup L_2}}$; thereby, we complete the proof of the claim.
\end{proof}
\subsection{Intersection cohomology of $\bM(V_5,3)$}\label{sec:5.2}
In this subsection, we compute the intersection cohomology of $\bM(V_5)$ by using Theorem \ref{mainthm}.
Let $X$ be a quasi-projective variety. For the (resp. intersection) Hodge-Deligne polynomial $\rE_c(X)(u,v)$ (resp. $\rIE_c(X)(u,v)$) for compactly supported (resp. intersection) cohomology of $X$, let
\[ \rP(X)=\rE_c(X)(-t,-t)\; (\mathrm{resp.}\;\rIP(X)=\rIE_c(X)(-t,-t))\]
be the \emph{virtual} (resp. intersection) Poincar\'e polynomial of $X$. See \cite[page 21]{Mun08} for the motivic properties of the virtual Poincar\'e polynomial. A map $\pi:X\lr Y$ is \emph{small} if for a locally closed stratification of $Y=\bigsqcup_i Y_i$ such that the restriction map $\pi|_{\pi^{-1}(Y_i)}:\pi^{-1}(Y_i)\lr Y_i$ is etale locally trivial, the inequality \[\dim \pi^{-1}(y)< \frac{1}{2}\mathrm{codim}_Y(Y_i)\] holds for each closed point $y\in Y_i$ except a dense open stratum of $Y$. Let $\pi: X\lr Y$ be a small map such that $X$ has at most finite group quotient singularities (more generally, \emph{rational homology manifold}). Then $\rP(X)=\rIP(Y)$ (\cite[Definition 6.6.1 and Theorem 6.6.3]{Max18}).
\begin{corollary}\label{corpoin}
The intersection cohomology of the stable map space $\bM(V_5)$ is given by
$$\rIP(\bM(V_5))=1+3t^2+8t^4+10t^6+8t^8+3t^{10}+t^{12}.$$
\end{corollary}
\begin{proof}
Because our blow-ups are weighted, the singularity in $\bM_3(V_5)$ is at most of the finite group quotient type.
On the other hand, the birational morphism $\widetilde{\Psi}^{\mathrm{I}}:\bM_3(V_5)\lr \bM(V_5)$ in Theorem \ref{mainthm} can be described by the result of \cite[Section 4.3 and Section 4.4]{CK11}. Recall that $IQ\cap \Delta=\bar{\Theta}_1(V_5)\subset\bar{\Theta}_2(V_5)$, where the latter space is isomorphic to $\bar{\Theta}_2(V_5)=IQ\cong \mathrm{Fl}(1,2;\CC^3)$ (Proposition \ref{blcenter2}). The exceptional locus $\bar{\Theta}_2^3(V_5)$ of the second blow up $\bM_3(V_5)\lr\bM_4(V_5)$ in Theorem \ref{mainthm} is a $\PP_{(1,2,2)}^2$-fibration over $\text{bl}_{\bar{\Theta}_1(V_5)}\bar{\Theta}_2(V_5)$. Also the exceptional divisor of the first blow-up map $\bM_4(V_5)\lr\bS(V_5)$ becomes a $\mathrm{bl}_{\PP^1}\PP_{(1,2,2,3,3)}^4$-fibration over $\bar{\Theta}_1(V_5)$.
Let us denote it by $\bar{\Theta}_1^3(V_5)$. Note that the center $\PP^1$ of the blow up map $\mathrm{bl}_{\PP^1}\PP_{(1,2,2,3,3)}^4\lr \PP_{(1,2,2,3,3)}^4$ is the projectivization of the normal space of the diagonal $IQ\cap \Delta$ in $IQ$. Hence the intersection part $\bar{\Theta}_2^3(V_5)\cap \bar{\Theta}_1^3(V_5)$ is a $\PP_{(1,2,2)}^2$-fibration over a $\PP^1$-fibration over $\bar{\Theta}_1(V_5)$. 
On the other hand, in $\bM(V_5)$, one can easily see that $\Gamma_0^1$ is isomorphic to a $\bM(\PP^1, 3)$-fibration over $\bar{\Theta}_1(V_5)$ and the closure $\bar{\Gamma}_{0}^{2}$ of $\Gamma_{0}^{2}$ is isomorphic to $\bM(\PP^1, 2)\cong \PP_{(1,2,2)}^2$-fibration over $\bar{\Theta}_2(V_5)$. Also, the intersection part $\Gamma_0^1\cap \bar{\Gamma}_{0}^{2}$ is isomorphic to a $\PP_{(1,2,2)}^2$-fibration over $\bar{\Theta}_1(V_5)$.
Since $\widetilde{\Psi}^{\mathrm{I}}(\bar{\Theta}_2^3(V_5)\cap \bar{\Theta}_1^3(V_5))=\Gamma_0^1\cap \bar{\Gamma}_{0}^{2}$ and $\widetilde{\Psi}$ is injective in the complement $\bM_3(V_5)\setminus \bar{\Theta}_2^3(V_5)\cap \bar{\Theta}_1^3(V_5)$ by its construction, the map $\widetilde{\Psi}^{\mathrm{I}}$ is a small one because
\[
\dim (\widetilde{\Psi}^{\mathrm{I}})^{-1} (y)=1<\frac{1}{2}\cdot(6-(2+1))=\frac{1}{2}(\dim \bM(V_5)-(\dim\PP_{(1,2,2)}^2+\dim (\bar{\Theta}_1(V_5)))),
\]
for each point $y\in \Gamma_0^1\cap \bar{\Gamma}_{0}^{2}$.
Therefore, $\rP(\bM_3(V_5))=\rIP(\bM(V_5))$ and
\[\begin{split}
\rP(\bM_4(V_5))&=\rP(\bS(V_5))+ (\rP(\PP_{(1,2,2,3,3)}^4)-1)\cdot \rP(\bar{\Theta}_1(V_5));\\
\rP(\bM_3(V_5))&=\rP(\bM_4(V_5))+(\rP(\PP_{(1,2,2)}^2)-1)\cdot[\rP( \mathrm{Fl}(1,2;\CC^3))+(\rP(\PP^1)-1)\cdot \rP(\PP^1)].
\end{split}\]
However, $\rP(\bS(V_5))=\rP(\mathrm{Gr}(2,5))=\frac{(1-t^{8})(1-t^{10})}{(1-t^2)(1-t^4)}$, $\rP(\PP^r)=\frac{1-t^{2r+2}}{1-t^2}$ and thus we obtain the result.
\end{proof}
\subsection*{Acknowledgements}
The author gratefully acknowledges the many helpful suggestions of In-Kyun Kim and SangHyeon Lee during the preparation of the paper. The author would like to thank the anonymous referee for valuable comments and suggestions to improve the quality of the paper.
Data sharing not applicable to this article as no datasets were generated or analysed during the current study.


\bibliographystyle{alpha}
\newcommand{\etalchar}[1]{$^{#1}$}

\end{document}